\documentclass[11pt]{amsart}

\usepackage{graphicx,pinlabel}

\newtheorem{theorem}{Theorem}[section]
\newtheorem{proposition}[theorem]{Proposition}

\newtheorem{corollary}[theorem]{Corollary}

\theoremstyle{definition}

\newtheorem{notation}[theorem]{Notation}

\numberwithin{equation}{section}

\newcommand{\diag}{\operatorname{diag}}

\newcommand{\Lk}{\operatorname{Lk}}

\newcommand{\B}{\operatorname{{\mathcal B}}}

\newcommand{\Q}{\operatorname{{\mathbb Q}}}

\newcommand{\Z}{\operatorname{{\mathbb Z}}}

\newcommand{\D}{\operatorname{{\mathcal D}}}

\begin{document}

\title[Middle tunnels]
{Middle tunnels by splitting}

\author{Sangbum Cho}
\address{Department of Mathematics Education\\
Hanyang University\\
Seoul 133-791\\
Korea}
\email{scho@hanyang.ac.kr}

\author{Darryl McCullough}
\address{Department of Mathematics\\
University of Oklahoma\\
Norman, Oklahoma 73019\\
USA}
\email{dmccullough@math.ou.edu}
\urladdr{www.math.ou.edu/$_{\widetilde{\phantom{n}}}$dmccullough/}
\thanks{The second author was supported in part by NSF grant DMS-0802424}


\date{\today}


\begin{abstract} For a genus-1 1-bridge knot in $S^3$, that is, a
$(1,1)$-knot, a middle tunnel is a tunnel that is not an upper or lower
tunnel for some $(1,1)$-position. Most torus knots have a middle tunnel,
and non-torus-knot examples were obtained by Goda, Hayashi, and
Ishihara. We generalize their construction and calculate the slope
invariants for the resulting middle tunnels. In particular, we obtain the
slope sequence of the original example of Goda, Hayashi, and Ishihara.
\end{abstract}

\maketitle

\section*{Introduction}
\label{sec:intro}

Genus-$2$ Heegaard splittings of the exteriors of knots in $S^3$ have been
a topic of considerable interest for several decades. They form a class
large enough to exhibit rich and interesting geometric behavior, but
restricted enough to be tractable. Traditionally such splittings are
discussed using the language of knot tunnels, which we will use from now
on.

The article \cite{CMtree} developed two sets of invariants that together
give a complete classification of all tunnels of all tunnel number 1
knots. One is a finite sequence of rational ``slope'' invariants, and the
other is a finite sequence of binary invariants. The latter sequence is
trivial exactly when the tunnel is a so-called $(1,1)$-tunnel, that is, a
tunnel that arises as the ``upper" or ``lower" tunnel of a genus-$1$
$1$-bridge position of the knot. In the language of \cite{CMtree}, the
$(1,1)$-tunnels are called semisimple, except for those which occur as the
upper and lower tunnels of a $2$-bridge knot and are called simple. The
tunnels which are not $(1,1)$-tunnels are called regular.

For quite a long time, the only known examples of knots having both regular
and $(1,1)$-tunnels were (most) torus knots, whose tunnels were classified
by M. Boileau, M. Rost, and H. Zieschang~\cite{B-R-Z} and independently by
Y. Moriah~\cite{Moriah}. Recently, another example was found by H. Goda and
C. Hayashi~\cite{Goda-Hayashi}. The knot is the Morimoto-Sakuma-Yokota
$(5,7,2)$-knot, and Goda and Hayashi credit H.\ Song with bringing it to
their attention. Like the torus knots, it has a $(1,1)$-position with two
associated semisimple tunnels, and a third ``middle'' tunnel which is
regular. A tunnel arc for the regular tunnel is shown in
Figure~\ref{fig:goda-hayashi}.
\begin{figure}\begin{center}
\includegraphics[width=40 ex]{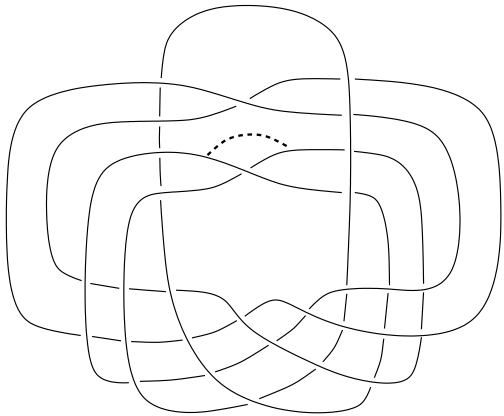}
\caption{The Morimoto-Sakuma-Yokota $(5,7,2)$-knot.
\newline The dotted line is a tunnel arc for its regular tunnel.}
\label{fig:goda-hayashi}
\end{center}
\end{figure}

A preliminary version of \cite{Goda-Hayashi} contained a gap in the
verification that the latter tunnel is not a $(1,1)$-tunnel: the authors
relied on Proposition~1.3 of (the nonetheless useful and important work)
\cite{Morimoto-Sakuma}, which turns out to be erroneous. As noted
in~\cite{Goda-Hayashi}, K. Ishihara~\cite{Ishihara} developed an algorithm
to compute the slope invariants of a tunnel using manipulation of families of
compressing disks in the associated Heegaard splitting, and successfully
applied it to compute the sequence of binary invariants of the tunnel,
sufficient to complete the proof that it is regular. In view of this, we
will refer to this example as the Goda-Hayashi-Ishihara tunnel. As noted
in~\cite{Goda-Hayashi}, a simple modification of their construction,
varying a nonzero integer parameter $n$, produces an infinite collection of
very similar examples.

In this paper, we analyze a general construction that produces all examples
directly obtainable by the geometric phenomenon that underlies the
Goda-Hayashi-Ishihara example. Moreover, we give an effective method to
compute the full set of slope invariants of any of these examples. We
illustrate it by computing the slope invariants of the
Goda-Hayashi-Ishihara example, and the binary invariants as well, verifying
Ishihara's calculation.

Here is a knot-theoretic description of the examples. As seen in
Figure~\ref{fig:GHI} below, the Morimoto-Sakuma-Yokota $(5,7,2)$-knot is
the band sum of two torus knots $T_{3, -4}$ and $T_{2, -3}$ lying in
concentric tori, by a (half-twisted) band running vertically between the
tori, with the tunnel represented by an arc cutting across the band. The
general example is an (arbitrarily twisted) band sum of two concentric
torus knots $T_{p+r, q+s}$ and $T_{r, s}$ (for certain allowable
combinations of $p$, $q$, $r$, and $s$). As we will see, in terms of our
theory this tunnel is obtained by a cabling construction starting from the
middle tunnel of the torus knot $T_{p+r, q+s}$.

For calculations, we need a very precise description. The general
construction, detailed in Section~\ref{sec:splitting} after preliminary
work in Sections~\ref{sec:torus_knots} and~\ref{sec:drop-lift}, is called
the splitting construction. There are four versions of it; each starting
with a so-called middle tunnel of a torus knot $K$, whose sequences of
invariants were calculated in~\cite{CMtorus}. Start with a torus knot $K$
contained in a standard torus $T$ in $S^3$, together with an arc in $T$
representing the middle tunnel of $K$. Regard $T$ as one level of a product
region $T\times I$. A tubular neighborhood of $K$, together with a
$1$-handle determined by the middle tunnel, is a genus-$2$ handlebody $H$
positioned ``horizontally'' in $T\times I$. Section~\ref{sec:drop-lift}
describes four disks, called the drop-$\rho$, lift-$\rho$, drop-$\lambda$,
and lift-$\lambda$ disks, and an isotopy that ``splits off" and either
``drops" or ``lifts" a solid torus from $H$. The solid torus is a
neighborhood of a certain torus knot $K'$ in another level of $T\times
I$. Inserting a disk called $\gamma_n$ into $H$, in a certain way, is a
cabling construction~\cite{CMtree} that produces the new tunnel (provided
that $n\neq 0$). Its associated knot is the sum of $K$ and $K'$, connected
by two vertical arcs in $T\times I$ positioned with $n$ half-twists. In
Section~\ref{sec:GHI_example} we give explicit versions of the splitting
construction that produce the Goda-Hayashi-Ishihara example and its mirror
image.

From the precise description, it is easy to read off the binary invariant
of this cabling construction. For the slope invariant, we set up a general
method in Sections~\ref{sec:first_general_slope_calculation}
and~\ref{sec:second_general_slope_calculation}. Besides adding the
transparency of abstraction, the setup will be used in~\cite{CMiterated} to
calculate the slope invariants obtained by an iteration of the splitting
construction, which we will discuss momentarily.
Section~\ref{sec:splitting_construction_slopes} uses the general method to
give the slopes in all cases of the splitting construction, and
Section~\ref{sec:GHI_calculation} illustrates them for the
Goda-Hayashi-Ishihara example.

Each tunnel obtained by the splitting construction is associated to a
$(1,1)$-position of its associated knot, and in
Section~\ref{sec:uuperlower} we explain how the method of~\cite{CMsemisimple}
allows an easy calculation of the slope invariants of its upper and lower
tunnels. As usual, we apply these to the Goda-Hayashi-Ishihara example.

We mentioned a further generalization of the splitting
construction. In~\cite{CMiterated}, we show how one can start with a tunnel
obtained by a splitting construction and carry out an iteration of similar
constructions, producing a much larger class of knots having both regular
and semisimple tunnels. Each of the four splitting constructions admits two
kinds of iteration sequences, giving eight versions of the iterated
construction. As with the splitting constructions, which allows variation
by any nonzero choice of $n$, each cabling in an iterated sequence can be
varied by a nonzero integer, producing an enormous number of possible
examples. Rather surprisingly to the authors, the setup of
Sections~\ref{sec:first_general_slope_calculation}
and~\ref{sec:second_general_slope_calculation} allows one to calculate the
slopes of all the cablings in the iterated construction.

We have already described most of the content of the paper, apart from the
first section below which establishes notation and reviews the method
from~\cite{CMtorus} for calculating the invariants of the middle tunnels of
torus knots. We have not included a review of the general theory, as the
original theory is detailed in~\cite{CMtree} and brief reviews are already
available in several of our articles. For the present paper, we would guess
that Section~1 of~\cite{CMgiant_steps} together with the review sections
of~\cite{CMsemisimple} form the best option for most readers.

\bigskip

\section{Middle tunnels of torus knots}
\label{sec:torus_knots}

Figure~\ref{fig:torus_knot_notation} shows a standard Heegaard torus $T$ in $S^3$,
and an oriented longitude-meridian pair $\{\ell,m\}$ which will be our
ordered basis for $H_1(T)$ and for the homology of a product neighborhood
$T\times I$. For a relatively prime pair of integers $(p,q)$, we denote by
$T_{p,q}$ a torus knot isotopic to a $(p,q)$-curve in $T$. In particular,
$\ell=T_{1,0}$ and $m=T_{0,1}$, also $T_{p,q}$ is isotopic in $S^3$ to
$T_{q,p}$, and $T_{-p,-q}=T_{p,q}$ since our knots are unoriented.
Figure~\ref{fig:torus_knot_notation} shows the knot~$T_{3,5}$.
\begin{figure}
\begin{center}
\labellist
\pinlabel $m$ at 152 55
\pinlabel $\ell$ at 190 55
\pinlabel \Large $T$ at 0 140
\small\pinlabel $T_{3,5}$ at 212 27
\endlabellist
\includegraphics[width=0.5\textwidth]{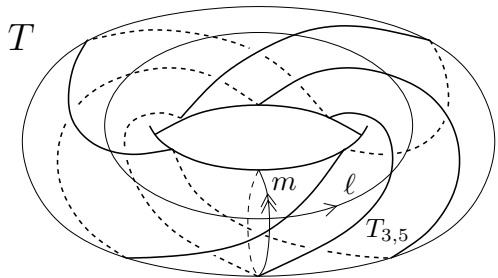}
\caption{$m$, $\ell$, and $T_{3,5}$.}
\label{fig:torus_knot_notation}
\end{center}
\end{figure}

\bigskip

We will sometimes but not always restrict attention to \textit{normalized}
torus knots, that is, to $T_{p,q}$ with $p> q\geq 2$. When allowing trivial
knots, we include $T_{n,1}$, $n\geq 1$, and $T_{1,0}$ as normalized torus
knots.

Any of our torus knot constructions or calculations can be reduced to this
case. To understand why, consider a product neighborhood $T\times [-1,1]$
of $T=T\times\{0\}$. There is an isotopy of $S^3$ that takes $T\times
\{s\}$ to $T\times \{-s\}$, interchanges $\ell$ and $m$, and moves
$T_{p,q}$ to $T_{q,p}$. Allowing such isotopies, we may always assume that
$|p|\geq |q|$. Since $T_{p,q}=T_{-p,-q}$, we may always assume further that
$p>0$, and if still $q<0$, we may apply a reflection of $S^3$ preserving
$T$ and taking $m$ to $-m$ and $\ell$ to $\ell$, so $T_{p,q}$ and
$T_{p,-q}$ are mirror images. The reflection
multiplies each slope invariant by $-1$.
As we will point out along the way, however, our
constructions and algebraic procedures always work, sometimes with some
simple modifications, for unnormalized torus knots.

We briefly recall the iterative construction of middle tunnels of torus
knots detailed in~\cite{CMtorus}, adapting the notation somewhat to suit
our current purposes. Figure~\ref{fig:drop-lambda_slope}(a) shows the
middle tunnel disk $\tau$ of a torus knot $K_\tau=T_{p+r,q+s}$
(in~\cite{CMtorus}, $p=p_1$, $r=p_2$, $q=q_1$, and $s=q_2$). Also seen
are the disks $\rho$ and $\lambda$ of the principal pair of $\tau$, whose
associated knots $K_\rho$ and $K_\lambda$ are torus knots $T_{p,q}$ and
$T_{r,s}$ respectively. Figure~\ref{fig:drop-lambda_slope}(a) shows the slope-$0$ separating
disk $\rho^0$ used to define $(\rho,\rho^0)$-coordinates. In general,
$\rho^0$ makes $(q+s)r$ turns around the handlebody, as indicated in the
drawing for the case $(q+s)r=2$.

Figure~\ref{fig:drop-lambda_slope}(b) shows a tunnel disk $\tau_U$, which
is obtained from $\tau$ by a cabling construction replacing $\rho$. It
meets $\rho$ in a single arc, and is disjoint from $\lambda$. As detailed
in~\cite{CMtorus}, and we hope is geometrically evident from
Figure~\ref{fig:drop-lambda_slope}(b), $\tau_U$ is the middle tunnel of
$T_{p+2r,q+2s}$. Figure~\ref{fig:drop-lambda_slope}(c) shows a similar
disk $\tau_L$ which is the middle tunnel for $T_{2p+r,2q+s}$, and is
obtained from $\tau$ by a cabling construction replacing $\lambda$. It
meets $\lambda$ in a single arc and is disjoint from~$\rho$.
\begin{figure}
\begin{center}
\labellist
\pinlabel \large (a) at 70 -5
\pinlabel $\rho$ at 110 200
\pinlabel $\lambda$ at 17 132
\pinlabel $\tau$ at 104 130
\pinlabel $\rho^0$ at 37 95
\pinlabel $\lambda$ at 116 78
\pinlabel $\rho$ at 25 10
\pinlabel $(q+s)r$ at 86 21
\pinlabel \large (b) at 226 -5
\pinlabel $\rho$ at 263 200
\pinlabel $\lambda$ at 170 132
\pinlabel $\tau$ at 256 130
\pinlabel $\tau_U$ at 192 91
\pinlabel $\lambda$ at 270 75
\pinlabel $\rho$ at 175 5
\pinlabel \large (c) at 380 -5
\pinlabel $\rho$ at 418 200
\pinlabel $\lambda$ at 323 131
\pinlabel $\tau$ at 409 127
\pinlabel $\tau_L$ at 340 63
\pinlabel $\lambda$ at 421 75
\pinlabel $\rho$ at 329 6
\endlabellist
\includegraphics[width=\textwidth]{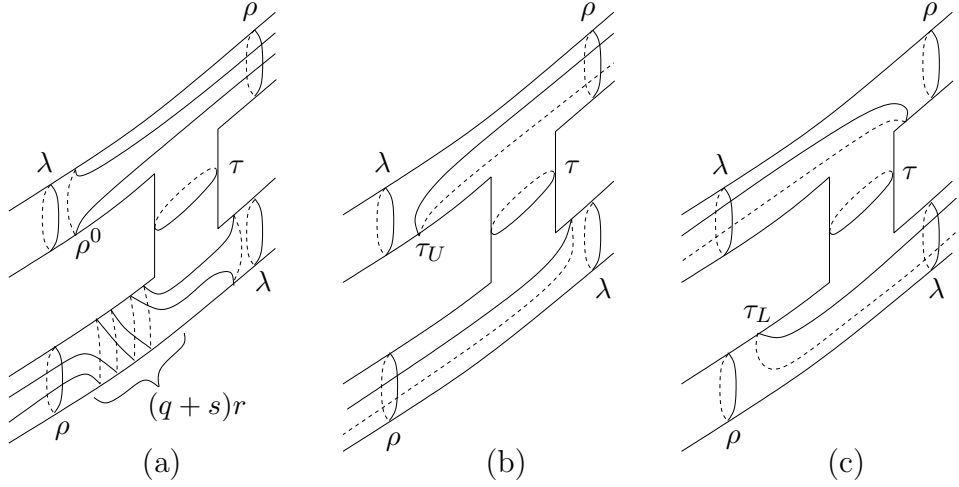}
\caption{The disks $\tau$, $\lambda$, $\rho$, $\rho^0$, $\tau_U$, and $\tau_L$.}
\label{fig:drop-lambda_slope}
\end{center}
\end{figure}

The notations here indicate the underlying algebra. Assume that $T_{p+r,q+s}$
is normalized, with $p+r > q+s   \geq 2$; as mentioned above, all other cases
can be reduced to this one. Write $(p+r)/(q+s)$ as a continued fraction
$[n_1,\ldots,n_k]$ with all $n_i$ positive.  Write
$U=\begin{pmatrix}1&1\\0&1\end{pmatrix}$ and
$L=\begin{pmatrix}1&0\\1&1\end{pmatrix}$. We call the matrix
\[M_{p+r,q+s}=(U\text{ or }L)^{n_k-1}\cdots U^{n_2}L^{n_1} =
\begin{pmatrix}p&q\\r&s\end{pmatrix}\]
the matrix \textit{associated to $T_{p+r,q+s}$.} As seen in \cite{CMtorus},
the knots $K_\rho$ and $K_\lambda$ are $T_{p,q}$ and $T_{r,s}$
respectively.

The associated matrix of
$T_{p+2r,q+2s}$ is
\[M_{p+2r,q+2s}=\begin{pmatrix}p+r&q+s\\r&s\end{pmatrix}=UM_{p+r,q+s}\ .\]
The principal pair of $\tau_U$ is $\{\lambda,\tau\}$, and passing from
$\tau$ to $\tau_U$ is a cabling construction which we call the
\textit{$U$-construction.} Similarly,
the associated matrix of $T_{2p+r,2q+s}$ is
\[M_{2p+r,2q+s}=\begin{pmatrix}p&q\\p+r&q+s\end{pmatrix}=LM_{p+r,q+s}\ ,\]
$\tau_L$ has principal pair
$\{\rho,\tau\}$ and is produced by the \textit{$L$-construction.}

As detailed in \cite{CMtorus}, a sequence of $U$- and $L$-constructions
producing the middle tunnel of $T_{p+r,q+s}$ (that is, the unique sequence
of cabling constructions producing the middle tunnel of $T_{p+r,q+s}$)
can be obtained as follows.
\begin{enumerate}
\item Start with the trivial knot positioned as $T_{1,1}$, whose
associated matrix is
\[M_{1+0,0+1}=\begin{pmatrix}1&0\\0&1\end{pmatrix}\ .\]
\item Perform $n_1$ $L$-constructions. The result is $T_{n_1+1,1}$, still a
trivial knot, but with associated matrix
\[L^{n_1}M_{1,1}=\begin{pmatrix}1&0\\n_1&1\end{pmatrix}\ .\]
\item Perform $n_2$ $U$-constructions, $n_3$ $L$-constructions, and so on,
except at the last step we perform only $n_k-1$ $U$- or $L$-constructions
(according as $k$ is even or odd). The resulting knot is $T_{p+r,q+s}$
and the effect of $U$- and $L$- constructions on the associated matrices
verifies that the associated matrix of $T_{p+r,q+s}$ is $M_{p+r,q+s}$.
\end{enumerate}

The construction we have discussed is for normalized $T_{p+r,q+s}$, but if
$q+s>p+r \geq 2$, the only difference is that $n_1=0$ and the first $n_2$
$U$-constructions produce trivial knots. If $p+r>0>q+s$, we may
perform a reflection to make both positive and proceed as before,
but the method can easily be adapted directly as
follows. To $T_{1,-1}$ we associate
\[M_{1,-1}=\begin{pmatrix}1&0\\0&-1\end{pmatrix}\ .\]
To find the matrix $M_{p+r,q+s}$ associated to $T_{p+r,q+s}$, we use the
continued fraction expression  $(p+r)/(q+s)=-[n_1,\ldots,n_k]$, with
$n_1\geq 0$ and $n_i\geq 1$ for $1\leq i\leq k$, to write
\[M_{p+r,q+s}=(U\text{ or }L)^{n_k-1}\cdots U^{n_2}L^{n_1}M_{1,-1}\ .\]
Starting with the trivial knot positioned as $T_{1,-1}$, perform
$n_1$ $L$-constructions, $n_2$ $U$-constructions, and so on, again ending
with $n_k-1$ ($U$ or $L$)-constructions to obtain the middle tunnel of
$T_{p+r,q+s}$.

For the next result, we introduce a useful notation.
\begin{notation}
The \textit{diagonal sum} of a $2\times 2$ matrix is the number
\[\diag \begin{pmatrix}a&b\\c&d\end{pmatrix} = ad+bc\ .\]
\end{notation}

The slopes of a $U$- or $L$-construction performed
on $T_{p+r,q+s}$ were obtained in~\cite{CMtorus}. We give them in the next
theorem, which for reference also summarizes some of the previous discussion.
\begin{theorem}\label{thm:torus_middle_tunnel_slopes}
Let $T_{p+r,q+s}$ be a torus knot, not $T_{\pm1,0}$ or $T_{0,\pm1}$. Applied to
$T_{p+r,q+s}$:
\begin{enumerate}
\item[(U)] The $U$-construction produces the middle tunnel of
$T_{p+2r,q+2s}$. Its slope is the slope of $\tau_U$ in
$(\rho,\rho^0)$-coordinates,
\[m_{\tau_U}=(p+r)s+(q+s)r =\diag(UM_{p+r,q+s})=\diag M_{p+2r,q+2s}\ .\]
\item[(L)] The $L$-construction produces the middle tunnel of
$T_{2p+r,2q+s}$. Its slope
is the slope of $\tau_L$ in $(\lambda,\lambda^0)$-coordinates,
\[m_{\tau_L}=p(q+s)+(p+r)q=\diag(LM_{p+r,q+s})=\diag M_{2p+r,2q+s}\ .\]
\end{enumerate}
\end{theorem}

In \cite{CMtorus}, only the normalized case is explicitly treated, but as
we have seen the procedures extend easily enough to the general case.

\section{Drop disks and lift disks}
\label{sec:drop-lift}

Certain disks, called the drop-$\lambda$, lift-$\lambda$, drop-$\rho$, and
lift-$\rho$ disks, will play a key role.

Figure~\ref{fig:drop-lambda} shows a picture of the drop-$\lambda$ disk, called
$\sigma$ there, and the knots $K_\tau=T_{p+r,q+s}$, $K_\rho=T_{p,q}$,
and $K_\lambda=T_{r,s}$. Figure~\ref{fig:drop-lambda}(a) shows $\sigma$ in the
standard picture of the middle tunnel, and Figure~\ref{fig:drop-lambda}(b)
shows an isotopic repositioning of the first configuration. In the latter,
$K_\tau$ and $K_\lambda$ are on concentric tori in a product neighborhood
$T\times I$ of the standard Heegaard torus in $S^3$, and the $1$-handle
with cocore $\sigma$ is a vertical $1$-handle connecting tubular
neighborhoods of these two knots. The term ``drop-$\lambda$'' is short for
``drop-$K_\lambda$'', motivated by the fact that a copy of $K_\lambda$ can be
dropped to a lower torus level.
\begin{figure}
\begin{center}
\labellist
\pinlabel \large (a) at 113 3
\pinlabel $\rho$ at 163 286
\pinlabel $\lambda$ at 28 185
\pinlabel $\tau$ at 91 143
\pinlabel $\sigma$ at 48 194
\pinlabel $\lambda$ at 172 114
\pinlabel $\sigma$ at 152 98
\pinlabel $\rho$ at 37 12
\pinlabel $K_\tau$ at -11 133
\pinlabel $K_\lambda$ at -11 37
\pinlabel $K_\lambda$ at 192 254
\pinlabel $K_\tau$ at 192 156
\pinlabel \large (b) at 337 3
\pinlabel $\lambda$ at 267 212
\pinlabel $\rho$ at 334 212
\pinlabel $\lambda$ at 400 212
\pinlabel $K_\tau$ at 441 189
\pinlabel $K_\rho$ at 441 167
\pinlabel $K_\rho$ at 227 100
\pinlabel $K_\lambda$ at 441 78
\pinlabel $\sigma$ at 279 133
\pinlabel $\tau$ at 267 56
\pinlabel $\tau$ at 401 56
\endlabellist
\includegraphics[width=0.8\textwidth]{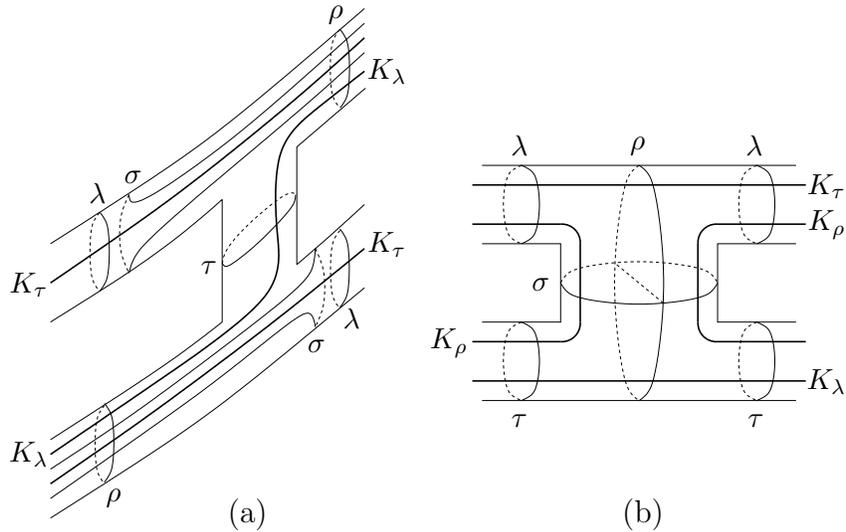}
\caption{The drop-$\lambda$ disk $\sigma$, first as seen in a
neighborhood of $K_\tau=T_{p+r,q+s}$ and the tunnel $\tau$, then after
dropping $K_\lambda=T_{r,s}$ and part of $K_\rho=T_{p,q}$.}
\label{fig:drop-lambda}
\end{center}
\end{figure}

The lift-$\lambda$ disk is similar, and is shown in
Figure~\ref{fig:lift-lambda}. The drop-$\rho$ and lift-$\rho$ disks are
similar, except that they cut across the upper copy of $\lambda$, travel
over the portion of the neighborhood of $T_{p+r,q+s}$ that does not contain
the drop-$\lambda$ disk, and cut across the lower copy of $\lambda$, while
staying disjoint from the copies of~$\rho$.
\begin{figure}
\begin{center}
\labellist

\pinlabel \large (a) at 113 3
\pinlabel $\rho$ at 163 286
\pinlabel $\lambda$ at 28 185
\pinlabel $\tau$ at 91 143
\pinlabel $\sigma$ at 48 194
\pinlabel $\lambda$ at 172 114
\pinlabel $\sigma$ at 152 98
\pinlabel $\rho$ at 37 12
\pinlabel $K_\tau$ at -11 133
\pinlabel $K_\lambda$ at -11 37
\pinlabel $K_\lambda$ at 192 254
\pinlabel $K_\tau$ at 192 156
\pinlabel \large (b) at 337 3
\pinlabel $\tau$ at 267 212
\pinlabel $\rho$ at 334 212
\pinlabel $\tau$ at 400 212
\pinlabel $K_\lambda$ at 441 189
\pinlabel $K_\rho$ at 441 167
\pinlabel $K_\rho$ at 227 100
\pinlabel $K_\tau$ at 441 78
\pinlabel $\sigma$ at 283 133
\pinlabel $\lambda$ at 267 56
\pinlabel $\lambda$ at 401 56

\endlabellist
\includegraphics[width=0.8\textwidth]{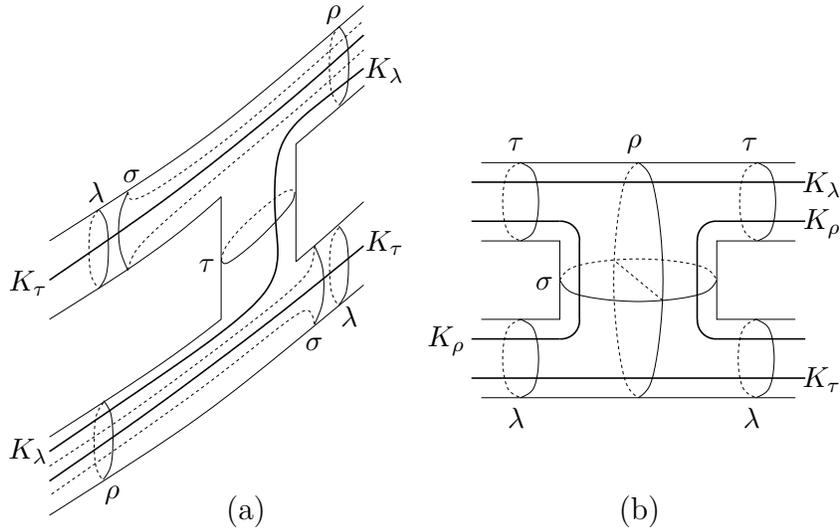}
\caption{The lift-$\lambda$ disk $\sigma$, first as seen in a
neighborhood of $K_\tau=T_{p+r,q+s}$ and the tunnel $\tau$, then after
lifting $K_\lambda=T_{r,s}$ and part of $K_\rho=T_{p,q}$.}
\label{fig:lift-lambda}
\end{center}
\end{figure}

\section{The splitting construction}
\label{sec:splitting}

We are now ready to present the basic construction. It is called the
\textit{splitting construction,} or just \textit{splitting,} because its
effect is to split a copy of $K_\rho=T_{p,q}$ or $K_\lambda=T_{r,s}$ off from
$K_\tau=T_{p+r,q+s}$, obtaining copies of these knots on two concentric
torus levels, then summing them together by a pair of arcs with some number of
twists.

There are actually four cases of the splitting construction. We begin with
the drop-$\lambda$ splitting. The first step was illustrated in
Figure~\ref{fig:drop-lambda}. Next, consider the disk $\gamma_n$ shown in
Figure~\ref{fig:gamma}. It is obtained from $\rho$ by $n$ right-handed
half-twists along $\sigma$. When $n<0$, the twists are left-handed, while
$\gamma_0=\rho$. The $\gamma_n$ are nonseparating in the genus-$2$
handlebody consisting of a tubular neighborhood of $K_\tau$ together with
the $1$-handle for its middle tunnel, since each $\gamma_n$ meets
$K_\tau$ in a single point.
\begin{figure}
\begin{center}
\labellist
\pinlabel $\lambda$ at 24 119
\pinlabel $\gamma_n$ at 80 119
\pinlabel $\lambda$ at 133 119
\pinlabel $K_\tau=T_{p+r,q+s}$ at 200 101
\pinlabel $K_\rho=T_{p,q}$ at 187 80
\pinlabel $\tau$ at 24 -7
\pinlabel $\tau$ at 134 -7
\pinlabel $K_\lambda=T_{r,s}$ at 187 10
\endlabellist
\includegraphics[width=0.36\textwidth]{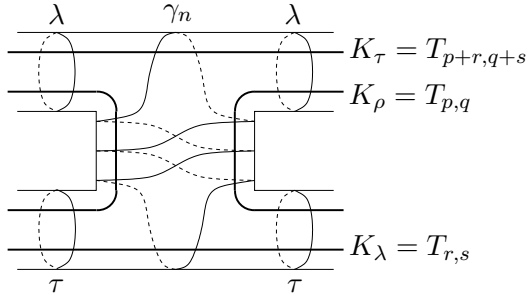}
\caption{The disk $\gamma_n$ is obtained from $\rho$ by $n$ right-handed
half-twists along $\sigma$. The case $n=3$ is shown here. For $n<0$, the
half-twists are left-handed, while $\gamma_0=\rho$.}
\label{fig:gamma}
\end{center}
\end{figure}

The disk $\gamma_n$ is a tunnel for the knot obtained by joining the
copies of $K_\tau$ and $K_\lambda$ in Figure~\ref{fig:drop-lambda} by a pair
of vertical arcs that have $n$ right-handed half-twists. Indeed, for $n\neq 0$
going from $\tau$ to $\gamma_n$ is a cabling construction replacing $\rho$,
so that the principal pair of $\gamma_n$ is $\{\lambda,\tau\}$. The case of
$n=0$ does not produce a cabling construction (that is, the resulting
tunnel would be $\rho$ so the principal path would have reversed
direction).

The lift-$\lambda$, drop-$\rho$, and lift-$\rho$ splittings are exactly
analogous, using the lift-$\lambda$, drop-$\rho$, and lift-$\rho$ disks
as $\sigma$ in the respective cases.

\section{The Goda-Hayashi-Ishihara tunnel}
\label{sec:GHI_example}

To illustrate the splitting construction, we will examine the first example
of a middle tunnel of a non-torus knot, which is due to H. Goda,
C. Hayashi, and K. Ishihara. The example was given by Goda and Hayashi
in~\cite{Goda-Hayashi}, indeed in an earlier preliminary version of that
article. In~\cite{Ishihara}, Ishihara developed a general algorithm to
compute slope and binary invariants and applied it to obtain the
principal path of the Goda-Hayashi tunnel, thereby proving that it is
regular. In principle, the algorithm could be used to obtain the slope
invariants, although this appears to be difficult.

The example is the Morimoto-Sakuma-Yokota knot of type
$(5,7,2)$~\cite{Morimoto-Sakuma-Yokota}. Goda and Hayashi credit H. Song
with bringing it to their attention. As noted
in~\cite{Morimoto-Sakuma-Yokota}, the knot can be moved into two concentric
Heegaard torus levels, apart from a pair of arcs that run between the
levels, as shown in Figure~\ref{fig:GHI}. The tunnel is seen as an arc in
the upper left-hand drawing, which is the knot. The remaining drawings show
two torus levels and a pair of connecting arcs running between them. On the
``upper'' level, the knot appears as a torus knot $T_{2,-3}$,
and on the ``bottom'' level as a torus knot $T_{3,-4}$. The pair of
connecting arcs has a single left-hand twist.

\begin{figure}
\begin{center}
\labellist
\endlabellist
\includegraphics[width=0.8\textwidth]{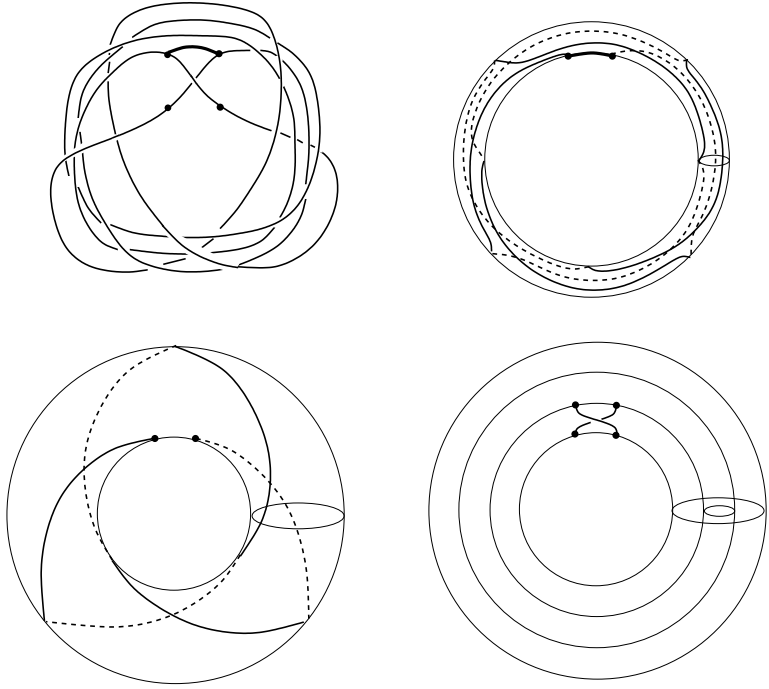}
\caption{The Goda-Hayashi-Ishihara example, seen as two torus levels
connected by a pair of arcs.}
\label{fig:GHI}
\end{center}
\end{figure}

This knot is obtained from the middle tunnel of $T_{3,-4}$ by a
lift-$\lambda$ splitting construction with $n=-1$. This allows us to find
its entire cabling sequence. We first calculate the continued fraction
expansion $-3/4=-[0,1,3]$ and use it to find that
\[M_{3,-4}=L^{(3-1)}U^1L^0\begin{pmatrix}1&0\\0&-1\end{pmatrix}
=\begin{pmatrix}1&-1\\2&-3\end{pmatrix}\ .\]
The cabling sequence is then as follows:
\begin{enumerate}
\item[1.] Starting with $T_{1,-1}$, a $U$-construction produces
$T_{1,-2}$ with associated matrix
\[M_{1,-2}=U\begin{pmatrix}1&0\\0&-1\end{pmatrix}
=\begin{pmatrix}1&-1\\0&-1\end{pmatrix}\ .\]
It is a trivial cabling construction since it produces a
trivial knot.
\item[2.] Next, an $L$-construction produces $T_{2,-3}$,
with associated matrix
\[M_{2,-3}=LM_{1,-2}=\begin{pmatrix}1&-1\\1&-2\end{pmatrix}\ .\]
According to
Theorem~\ref{thm:torus_middle_tunnel_slopes},
the slope of this cabling is $\diag(M_{2,-3}) = -3$, so the simple slope is
$[-1/3]=[2/3]$
(the simple slope, used for the first
nontrivial cabling construction in the cabling sequence, is the reciprocal
of the slope modulo $\Q/\Z$).
\item[3.] Another $L$-construction produces the middle tunnel of
$T_{3,-4}$,
with associated matrix
\[M_{3,-4}=LM_{2,-3}
=\begin{pmatrix}1&-1\\2&-3\end{pmatrix}\ .\]
According to Theorem~\ref{thm:torus_middle_tunnel_slopes},
This time the slope is $\diag(M_{3,-4}) = -5$.
\item[4.] A lift-$\lambda$ splitting lifts a copy of
$K_\lambda=T_{2,-3}$ to the top level, and using $\gamma_{-1}$ as the
tunnel disk puts one left-hand twist in the two vertical strands, producing
the Goda-Hayashi-Ishihara knot. A tunnel arc for $\gamma_{-1}$ runs
horizontally between the two vertical strands, so is the
Goda-Hayashi-Ishihara tunnel.
\end{enumerate}
This construction proves that the tunnel is regular, since the
$L$-constructions replace $\lambda$, while the lift-$\lambda$ construction
replaces $\rho$. In Section~\ref{sec:GHI_calculation}, we will see that the
final splitting construction in its cabling sequence has slope~$-19$,
giving the full principal path of the Goda-Hayashi-Ishihara tunnel shown in
Figure~\ref{fig:GHI_path} below.

As remarked in Section~\ref{sec:torus_knots}, one can also maneuver so that
the splitting takes place on a normalized torus knot $T_{p+r,q+s}$.
Starting with the Goda-Hayashi-Ishihara knot, apply an isotopy of $S^3$
that interchanges the meridian and longitude of the level tori. It inverts
the order as well, putting $T_{4,-3}$ on the upper level and $T_{3,-2}$ on
the bottom level, while preserving the tunnel. The vertical arcs still have
a left-handed twist. Next, apply an orientation-reversing diffeomorphism
that fixes the longitudes of the Heegaard tori and reflects the meridians,
after which the top level is $T_{4,3}$ and the bottom level is
$T_{3,2}$. In addition, the two vertical arcs now have one right-handed
half-twist, rather than left-handed, since the reflection reverses the
sense of the twist. The orientation-reversing diffeomorphism negates the
values of the slope invariants, and does not change the binary
invariants. Since the continued fraction expansion of $4/3$ is $[1,3]$, the
torus knot $T_{4,3}=T_{3+1,2+1}$ has associated matrix
\[M_{4,3}=U^2L = \begin{pmatrix}3&2\\1&1\end{pmatrix}\ .\]
Starting with $T_{1,1}$, one $L$-construction followed by two
$U$-constructions produces the middle tunnel of $T_{4,3}$, with
$K_\rho=T_{3,2}$ and $K_\lambda=T_{1,1}$. Now a drop-$\rho$ splitting drops
a copy of $T_{3,2}$ to the lower level, and using $\gamma_1$ puts the
right-hand half-twist in the vertical strands. The slope and binary
invariants can be calculated, as we will see in
Section~\ref{sec:GHI_calculation} below, and the slope invariants negated
to obtain the slope invariants of the original unreflected example.

As noted in \cite{Goda-Hayashi}, infinitely many similar examples are
obtained by changing the number of twists of the vertical strands, that is,
by different choices of~$\gamma_n$.

\section{The first general slope calculation}
\label{sec:first_general_slope_calculation}

In order to understand the slope invariants of tunnels resulting from the
splitting constructions, we must calculate the slopes of the disks
$\gamma_n$ in certain coordinates. For this, one needs the slopes of the
drop- and lift-disks. In fact, there is a general slope calculation that
covers all four cases (as well as additional cases that will arise
in~\cite{CMiterated}). In this section, we present this general slope
calculation, and in the next section, we present the calculation of the
slopes of disks~$\gamma_n$.

Consider the setup illustrated in Figure~\ref{fig:first_general}(a). The
first drawing shows tubular neighborhoods of two (oriented) knots $K_U$ and
$K_L$, contained in a product neighborhood $T\times I$ of a Heegaard torus
$T$ of $S^3$. The neighborhoods are connected by a vertical $1$-handle to
yield a genus-$2$ handlebody $H$. In our applications, $H$ will always be
unknotted, although that is not needed for the calculations of this and the
next section.
\begin{figure}
\begin{center}
\labellist
\pinlabel \large (a) at 67 -7
\pinlabel \large (b) at 217 -7
\pinlabel \large (c) at 370 -7
\scriptsize\pinlabel $D_U^-$ at 20 123
\scriptsize\pinlabel $D$ at 65 123
\scriptsize\pinlabel $D_U^+$ at 110 123
\scriptsize\pinlabel $D_U^-$ at 170 123
\scriptsize\pinlabel $D_U^+$ at 261 123
\scriptsize\pinlabel $D_U^-$ at 320 123
\scriptsize\pinlabel $D^0$ at 368 123
\scriptsize\pinlabel $D_U^+$ at 418 123
\scriptsize\pinlabel $K_U$ at -8 105
\scriptsize\pinlabel $K_U^0$ at 142 95
\scriptsize\pinlabel $K_L^0$ at 142 39
\scriptsize\pinlabel $K_L$ at -8 31
\small\pinlabel $\sigma$ at 28 68
\scriptsize\pinlabel $D_L^-$ at 20 12
\scriptsize\pinlabel $D_L^+$ at 111 12
\scriptsize\pinlabel $D_L^-$ at 171 12
\scriptsize\pinlabel $D_L^+$ at 262 12
\scriptsize\pinlabel $D_L^-$ at 322 12
\scriptsize\pinlabel $D_L^+$ at 419 12
\endlabellist
\includegraphics[width=0.93\textwidth]{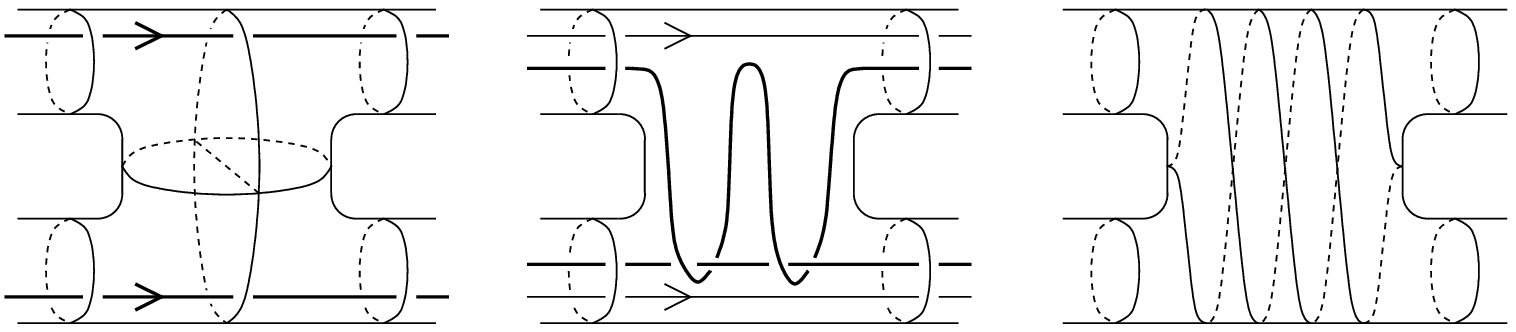}
\caption{The setup for the first general slope calculation.}
\label{fig:first_general}
\end{center}
\end{figure}

We interpret $K_U$ as the ``upper'' knot, contained in $T\times (3/4,1]$,
and $K_L$ as the ``lower'' knot, contained in $T\times [0,1/4)$.  The
vertical $1$-handle with cocore $\sigma$ is assumed to run between $T\times
\{1/4\}$ and $T\times \{3/4\}$, with the separating disk $\sigma$ being its
intersection with~$T\times \{1/2\}$.

The homology group $H_1(T\times I)\cong H_1(T)$ will have ordered basis the
oriented longitude and meridian $\ell$ and $m$ shown in
Figure~\ref{fig:torus_knot_notation}. Our linking convention is that
$\Lk(m\times \{1\},\ell\times \{0\})=+1$. Now, suppose that $K_U$ represents
$(\ell_U,m_U)$ and $K_L$ represents $(\ell_L,m_L)$ in $H_1(T\times I)$.
Since $\Lk(m\times \{0\},\ell\times \{1\})=0$, we have
$\Lk(K_U,K_L)=m_U\ell_L$.

The disks $D_U^+$ and $D_U^-$ are parallel in $H$, as are the disks $D_L^+$
and $D_L^-$, and these four disks bound a ball $B$ seen in
Figure~\ref{fig:first_general}. Our task is to compute the slope of
$\sigma$ in $(D,D^0)$-coordinates. Here, $D$ is a slope disk in $B$ seen in
Figure~\ref{fig:first_general}(a), and $D^0$ is the slope-$0$
disk in $B$ that meets $D$ in a single arc and separates $H$ into two solid
tori with linking number~$0$.

Figure~\ref{fig:first_general}(b) shows core knots $K_U^0$ and $K_L^0$ of
the complementary solid tori of $D^0$. They are like $K_U$ and $K_L$,
except that they have $\Lk(K_U,K_L)$ right-handed full twists in this
picture. Provided that the orientations of $K_U$ and $K_L$ appear from
left-to-right, as indicated in Figure~\ref{fig:first_general}(b), each
right-handed twist changes the linking number by $-1$.
Figure~\ref{fig:first_general} is drawn for the case $\Lk(K_U,K_L)=2$, so
there are two right-handed twists and $\Lk(K_U^0,K_L^0)=0$.

If one uses the opposite linking convention that $\Lk(m\times
\{1\},\ell\times \{0\})=-1$, then $\Lk(K_U,K_L)$ is negated but the effect
of a right-handed twist is also negated. Thus $K_U^0$ and $K_L^0$ are the
same whatever independent of the linking convention, and the slope-$0$ disk
$D^0$, shown in Figure~\ref{fig:first_general}(c), is well-defined.

We are now ready to compute the slope of $\sigma$ in $(D,D^0)$-coordinates.
Figure~\ref{fig:first_general_cabling_arcs}(a) shows a cabling arc
$\alpha(D^0)$, that is, an arc in $B\cap \partial H$ connecting two
frontier disks and disjoint from $D^0$. In this instance, the disk is the
$D^0$ shown in Figure~\ref{fig:first_general}(c), so $\alpha(D^0)$ makes
two turns around $B$ in the direction shown. Also seen in
Figure~\ref{fig:first_general_cabling_arcs}(a) is a cabling arc
$\alpha(\sigma)$ for~$\sigma$.
\begin{figure}
\begin{center}
\labellist
\pinlabel \large (a) at 104 228
\scriptsize\pinlabel $D_U^-$ at 25 421
\scriptsize\pinlabel $D_U^+$ at 182 421
\scriptsize\pinlabel $\alpha(\sigma)$ at 103 395
\scriptsize\pinlabel $\alpha(D^0)$ at 103 259
\scriptsize\pinlabel $D_L^-$ at 25 258
\scriptsize\pinlabel $D_L^+$ at 182 257
\pinlabel \large (b) at 104 5
\scriptsize\pinlabel $D_U^-$ at 25 199
\scriptsize\pinlabel $D_U^+$ at 181 199
\scriptsize\pinlabel $\alpha(\sigma)$ at 104 199
\scriptsize\pinlabel $\alpha(D^0)$ at 104 59
\scriptsize\pinlabel $D_L^-$ at 25 34
\scriptsize\pinlabel $D_L^+$ at 179 34
\pinlabel \large (c) at 393 7
\tiny\pinlabel $D_L^+$ at 306 398
\tiny\pinlabel $D_L^-$ at 355 398
\tiny\pinlabel $D_L^+$ at 406 398
\tiny\pinlabel $D_L^-$ at 457 398
\tiny\pinlabel $D_U^+$ at 331 322
\tiny\pinlabel $D_U^-$ at 381 322
\tiny\pinlabel $D_U^+$ at 431 322
\tiny\pinlabel $D_U^-$ at 482 322
\tiny\pinlabel $D_L^+$ at 331 269
\tiny\pinlabel $D_L^-$ at 381 269
\tiny\pinlabel $D_L^+$ at 431 269
\tiny\pinlabel $D_L^-$ at 482 269
\tiny\pinlabel $D_U^+$ at 306 215
\tiny\pinlabel $D_U^-$ at 355 215
\tiny\pinlabel $D_U^+$ at 406 215
\tiny\pinlabel $D_U^-$ at 457 215
\tiny\pinlabel $D_L^+$ at 306 165
\tiny\pinlabel $D_L^-$ at 355 165
\tiny\pinlabel $D_L^+$ at 406 165
\tiny\pinlabel $D_L^-$ at 457 165
\tiny\pinlabel $D_U^+$ at 331 116
\tiny\pinlabel $D_U^-$ at 381 116
\tiny\pinlabel $D_U^+$ at 431 116
\tiny\pinlabel $D_U^-$ at 482 116
\tiny\pinlabel $D_L^+$ at 331 63
\tiny\pinlabel $D_L^-$ at 381 63
\tiny\pinlabel $D_L^+$ at 431 63
\tiny\pinlabel $D_L^-$ at 482 63
\endlabellist
\includegraphics[width=0.9\textwidth]{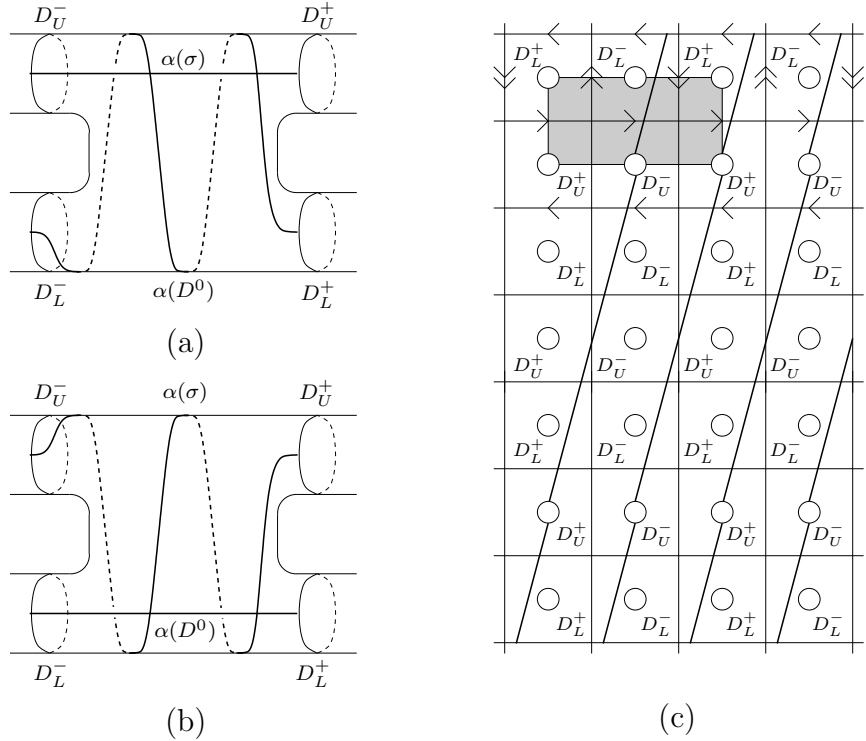}
\caption{The cabling arc $\alpha(\sigma)$ and some of its lifts.}
\label{fig:first_general_cabling_arcs}
\end{center}
\end{figure}

Figure~\ref{fig:first_general_cabling_arcs}(b) is simply
Figure~\ref{fig:first_general_cabling_arcs}(a) redrawn so that
$\alpha(D^0)$ appears horizontal. This moves $\alpha(\sigma)$ to an arc
that makes two turns in the opposite direction from the turns of
$\alpha(D^0)$ in Figure~\ref{fig:first_general_cabling_arcs}(a), that is,
the two right-handed turns of $\alpha(D^0)$ become two left-handed turns
of $\alpha(\sigma)$.

Figure~\ref{fig:first_general_cabling_arcs}(c) shows part of the covering
space of $\Sigma=B\cap \partial H$ seen in Figure~8 of ~\cite{CMsemisimple} (originally, in Figure~7 of~\cite{CMtree}). The
shaded region is a fundamental domain, and each boundary circle of the
covering space double covers the indicated boundary circle of $\Sigma$. The
lifts of $\alpha(D^0)$ are horizontal arcs connecting inverse image circles
of $D_U^+$ to inverse image circles of $D_U^-$. The lifts of the cabling arc
of $\sigma$ appear as line segments connecting the inverse image circles
for $D_L^+$ to inverse image circles for $D_L^-$. In the case shown, those
segments have slope pair $[1,4]$, as each left-handed turn of
$\alpha(\sigma)$ around $B$ produces two vertical units of rise in the
lift. In general, if $\alpha(D^0)$ made $R$ right-handed twists, the slope
pair of $\sigma$ is $[1,2R]$, that is, its slope is $2R/1$. Since $R$ was
$\Lk(K_U,K_L)$, this yields our first general slope calculation. Assuming
that the orientation of $K_U$ and $K_L$ is from left to right in the
figures we have discussed, and that we use our linking convention
$\Lk(m\times \{1\},\ell\times \{0\})=1$, we have
\begin{proposition}\label{prop:first_general_slope_calculation}
In Figure~\ref{fig:first_general}(a), the slope of $\sigma$ in
$(D,D^0)$-coordinates is $2\Lk(K_U,K_L)$. Consequently, if $K_U$ represents
$(\ell_U,m_U)$ and $K_L$ represents $(\ell_L,m_L)$ in $H_1(T\times I)$,
then the slope of $\sigma$ is $2m_U\ell_L$.
\end{proposition}

As an immediate consequence:
\begin{corollary}\label{coro:splitting_disk_slopes}
The slopes of the splitting disks are as follows:
\begin{enumerate}
\item[(a)] In $(\rho,\rho^0)$-coordinates,
the drop-$\lambda$ disk has slope $2r(q+s)$.
\item[(b)] In $(\rho,\rho^0)$-coordinates,
the lift-$\lambda$ disk has slope $2s(p+r)$.
\item[(c)] In $(\lambda,\lambda^0)$-coordinates,
the drop-$\rho$ disk has slope $2p(q+s)$
\item[(d)] In $(\lambda,\lambda^0)$-coordinates,
the lift-$\rho$ disk has slope $2q(p+r)$.
\end{enumerate}
\end{corollary}

\section{The second general slope calculation}
\label{sec:second_general_slope_calculation}

It remains to obtain the slope of $\gamma_n$. Figure~\ref{fig:gamma_slopes}
illustrates the calculation.
\begin{figure}
\begin{center}
\labellist
\pinlabel \large (a) at 104 115
\scriptsize\pinlabel $D_U^-$ at 26 316
\scriptsize\pinlabel $D_U^+$ at 193 316
\scriptsize\pinlabel $\alpha(\sigma)$ at 107 290
\scriptsize\pinlabel $\alpha(\gamma_3)$ at 107 197
\scriptsize\pinlabel $D_L^-$ at 26 144
\scriptsize\pinlabel $D_L^+$ at 193 144
\normalsize\pinlabel \large (b) at 398 3
\tiny\pinlabel $D_L^-$ at 336 68
\tiny\pinlabel $D_L^+$ at 383 68
\tiny\pinlabel $D_L^-$ at 434 68
\tiny\pinlabel $D_L^+$ at 484 68
\tiny\pinlabel $D_U^-$ at 313 120
\tiny\pinlabel $D_U^+$ at 362 120
\tiny\pinlabel $D_U^-$ at 411 120
\tiny\pinlabel $D_U^+$ at 460 120
\tiny\pinlabel $D_L^-$ at 336 168
\tiny\pinlabel $D_L^+$ at 383 168
\tiny\pinlabel $D_L^-$ at 434 168
\tiny\pinlabel $D_L^+$ at 484 168
\tiny\pinlabel $D_U^-$ at 313 220
\tiny\pinlabel $D_U^+$ at 362 220
\tiny\pinlabel $D_U^-$ at 411 220
\tiny\pinlabel $D_U^+$ at 460 220
\tiny\pinlabel $D_L^-$ at 336 269
\tiny\pinlabel $D_L^+$ at 383 269
\tiny\pinlabel $D_L^-$ at 434 269
\tiny\pinlabel $D_L^+$ at 484 269
\tiny\pinlabel $D_U^-$ at 313 319
\tiny\pinlabel $D_U^+$ at 362 319
\tiny\pinlabel $D_U^-$ at 411 319
\tiny\pinlabel $D_U^+$ at 460 319
\tiny\pinlabel $D_L^-$ at 336 367
\tiny\pinlabel $D_L^+$ at 383 367
\tiny\pinlabel $D_L^-$ at 434 367
\tiny\pinlabel $D_L^+$ at 484 367
\tiny\pinlabel $D_U^-$ at 313 419
\tiny\pinlabel $D_U^+$ at 362 419
\tiny\pinlabel $D_U^-$ at 411 419
\tiny\pinlabel $D_U^+$ at 460 419
\endlabellist
\includegraphics[width=0.9\textwidth]{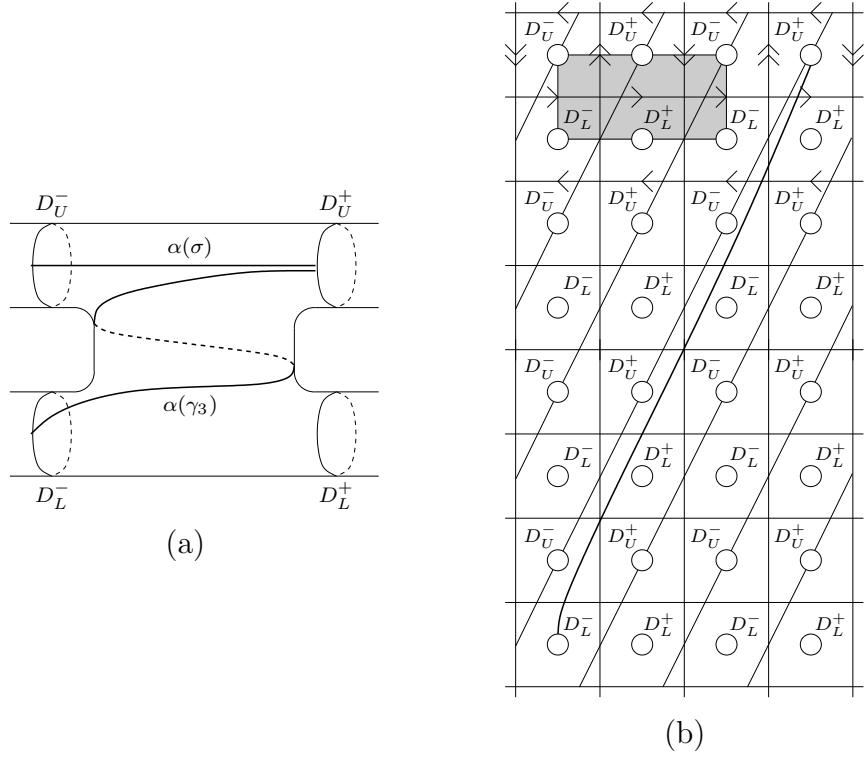}
\caption{Calculation of the slope of $\gamma_n$.}
\label{fig:gamma_slopes}
\end{center}
\end{figure}
Figure~\ref{fig:gamma_slopes}(a) shows the ball $B$ from
Figures~\ref{fig:first_general}(a)
and~\ref{fig:first_general_cabling_arcs}(a), bounded by the disks $D_U^-$,
$D_U^+$, $D_L^-$, and $D_L^+$.  The arc $\alpha(\sigma)$
connecting $D_U^-$ and $D_U^+$ is a
cabling arc for $\sigma$, and the arc $\alpha(\gamma_3)$
connecting $D_L^-$ and $D_U^+$ is a
cabling arc for $\gamma_3$. In general, one of the cabling arcs for $\gamma_n$
connects $\D_L^-$ to either $D_U^-$ or $D_U^+$ according as $n$ is
even or odd.

Again we use the covering space from
Figure~\ref{fig:first_general_cabling_arcs}(c).  As before, the lifts of
$\alpha(\sigma)$ appear as line segments connecting inverse image circles
for $D_U^+$ to inverse image circles for $D_U^-$. In the case shown in
Figure~\ref{fig:gamma_slopes}(b), those segments have slope pair $[1,2]$,
while Proposition~\ref{prop:first_general_slope_calculation} show that in
general, the slope pair of the lifts of $\alpha(\sigma)$
is~$[1,2\Lk(K_U,K_L)]=[1,m_\sigma]$, where $m_\sigma$ is the slope of
$\sigma$ in $(D,D^0)$-coordinates.

Figure~\ref{fig:gamma_slopes}(b) also shows a lift of
$\alpha(\gamma_3)$. Since $\alpha(\gamma_3)$ or in general
$\alpha(\gamma_n)$ is disjoint from $\alpha(\sigma)$, the lift cannot cross
the line segments that are lifts of $\alpha(\sigma)$. Each right-hand
half-twist of $\gamma_n$ corresponds to a right-hand half-twist of
$\alpha(\gamma_n)$, and an upward displacement of the lift of
$\alpha(\gamma_n)$ that runs roughly parallel to one of the segments that
is a lift of $\alpha(\sigma)$. Thus in general the slope pair of the lifts
of the cabling arc for $\gamma_n$ is $[0,1]+n[1,m_\sigma]$. Consequently
the slope pair of $\gamma_n$ is $[n,1+nm_\sigma]$, and its slope is
$(1+nm_\sigma)/n=m_\sigma+1/n$. This gives our second general slope
calculation. Again with our usual orientation and linking conventions:
\begin{proposition}\label{prop:second_general_slope_calculation}
The slope of $\gamma_n$ in $(D,D^0)$-coordinates is $m_\sigma+1/n$.
\end{proposition}

\section{Slopes for the splitting construction}
\label{sec:splitting_construction_slopes}

Corollary~\ref{coro:splitting_disk_slopes} and
Proposition~\ref{prop:second_general_slope_calculation} give immediately
the slopes of the four splitting constructions:
\begin{proposition}\label{prop:sigma_slopes}
For the torus knot $T_{p+r,q+s}$:
\begin{enumerate}
\item[(a)] A drop-$\lambda$ splitting has slope $2r(q+s)+1/n$.
\item[(b)] A lift-$\lambda$ splitting has slope $2s(p+r)+1/n$.
\item[(c)] A drop-$\rho$ splitting has slope $2p(q+s)+1/n$.
\item[(d)] A lift-$\rho$ splitting has slope $2q(p+r)+1/n$.
\end{enumerate}
\end{proposition}

With the exception of a few phenomena described in the next theorem,
splitting constructions on nontrivial torus knots produce distinct tunnels.
\begin{theorem}\label{thm:coincidences}
Suppose that two splitting constructions on a nontrivial normalized torus
knot produce the same tunnel. Then both splittings are obtained from
the same torus knot, and either
\begin{enumerate}
\item[(a)] one is a drop-$\lambda$ splitting with $n=1$ and the other is a
lift-$\lambda$ splitting with $n=-1$, and the tunnel is the middle tunnel of
$T_{p+2r,q+2s}$, or
\item[(b)] one is a drop-$\rho$ splitting with $n=-1$ and the other is a
lift-$\rho$ splitting with $n=1$, and the tunnel is the middle tunnel of
$T_{2p+r,2q+s}$, or
\item[(c)] the knot in normalized form is $T_{2r+1,2}$, and
the splittings are either
\begin{enumerate}
\item[(i)] the lift-$\lambda$ and lift-$\rho$ splittings with the same value of~$n$, or
\item[(ii)] the lift-$\lambda$ splitting with $n=1$ and the drop-$\rho$ splitting with $n=-1$, or
\item[(iii)] the drop-$\lambda$ splitting with $n=1$ and the lift-$\rho$ splitting with $n=-1$.
\end{enumerate}
\end{enumerate}
\end{theorem}
\noindent For the trivial normalized torus knots $T_{p,1}$, $p\geq 1$, one
can quickly work out the results of all possible splittings by using
Proposition~\ref{prop:sigma_slopes}. They are the simple tunnels having
slope invariant $[n/(2kn+1)]$, $k\geq 0$.

Before proving Theorem~\ref{thm:coincidences}, we identify the tunnels and
knots that arise from the multiple splittings that it classifies:
\begin{corollary}\label{coro:multiple_splittings}
The following are the tunnels that arise from distinct splittings on some
nontrivial torus knot:
\begin{enumerate}
\item[(a)] The middle tunnel of each normalized torus knot $T_{a,b}$ with
$b\geq 4$ arises from exactly two splittings:
\begin{enumerate}
\item[(i)] If the tunnel arises from a $U$-construction on $T_{p+r,q+s}$,
and hence is the middle tunnel of $T_{p+2r,q+2s}$, then it arises from
$T_{p+r,q+s}$ using either a drop-$\lambda$ splitting with $n=1$ or a
lift-$\lambda$ splitting with $n=-1$.
\item[(ii)] If the tunnel arises from an $L$-construction on $T_{p+r,q+s}$,
and hence is the middle tunnel of $T_{2p+r,2q+s}$, then it arises from
$T_{p+r,q+s}$ using either a drop-$\rho$ splitting with $n=-1$ or a
lift-$\rho$ splitting with $n=1$.
\end{enumerate}
\item[(b)] For each $r\geq 1$ and nonzero integer $n$ with $|n|\geq 2$,
there is a semisimple tunnel of a non-torus $3$-bridge knot that arises
from exacly two distinct splittings on $T_{2r+1,2}$: lift-$\lambda$ and
lift-$\rho$ splittings with the value~$n$. It has slope sequence
$[1/(2r+1)]$, $4n+2+1/n$.
\item[(c)] For each torus knot $T_{3r+1,3}$, $r\geq 1$,
the middle tunnel, which is semisimple, arises from three distinct
splittings on $T_{2r+1,2}$: lift-$\lambda$ and lift-$\rho$ splittings with
$n=1$, and a drop-$\rho$ splitting with $n=-1$.
\item[(d)] For each torus knot $T_{3r+2,3}$, $r\geq 1$,
the middle tunnel, which is semisimple, arises from three distinct
splittings on $T_{2r+1,2}$: lift-$\lambda$ and lift-$\rho$ splittings with
$n=-1$, and a drop-$\lambda$ splitting with $n=1$.
\end{enumerate}
\end{corollary}
\begin{proof}
Case~(a) just describes cases~(a) and~(b) of
Theorem~\ref{thm:coincidences}. In cases~(b), (c), and~(d), the
tunnels are semisimple since they result from only two cablings.  Also,
since the tunnels are constructed by cabling sequences of length~$2$,
Theorem~6.1 of~\cite{CMbridge} shows that the associated knots have bridge
number at most~$3$.

In Theorem~\ref{thm:coincidences}(c)(i),
Proposition~\ref{prop:sigma_slopes} finds the cabling sequences to be
$[1/(2r+1)]$, $4r+2+1/n$. The associated knots are not $2$-bridge since for
tunnels of $2$-bridge knots every slope invariant after the first is of the
form $\pm 2 + 1/n$, see \cite[Section~15]{CMtree}. For $n$ with $|n|>1$,
these give case~(b). Since the second slope invariant is not integral,
these are not torus knots~\cite[Section 6]{CMtorus}. Those with $|n|=1$
will appear in cases~(c) and~(d).

In Theorem~\ref{thm:coincidences}(c)(ii) and (c)(iii), the slope sequences
are respectively $[1/(2r+1)]$, $4r+1$ and $[1/(2r+1)]$, $4r+3$, and the
algorithm of \cite[Section 6]{CMtorus} identifies these as the middle
tunnels of the torus knots $T_{3r+1,3}$ and $T_{3r+2,3}$ respectively.
These give cases~(c) and~(d).
\end{proof}

\begin{proof}[Proof of Theorem~\ref{thm:coincidences}]
Consider a normalized torus knot $T_{p+r,q+s}$, for $p+r>q+s\geq 2$, with
associated matrix
\[M_{p+r,q+s}=(U\text{ or }L)^{n_k-1}\cdots U^{n_2}L^{n_1} =
\begin{pmatrix}p&q\\r&s\end{pmatrix}\ .\]
We recall from Section~\ref{sec:torus_knots} that the sequence of $U$- and
$L$-cablings producing the middle tunnel of $T_{p+r,q+s}$ is determined by
the positive integer continued fraction expansion $[n_1,\ldots,n_k]$ of
$(p+r)/(q+s)>1$.  Since this expansion is unique, apart from the ambiguity
that $[n_1,\ldots,n_k,1]=[n_1,\ldots,n_k+1]$, middle tunnels of nontrivial
torus knots have the same principal path only when they are the same
tunnel. Since the principal path of a splitting construction is a
continuation of the principal path of the middle tunnel on which it is
performed, splitting constructions on distinct middle tunnels of torus
knots cannot produce the same tunnel. So we need only consider a pair of
splitting constructions applied to the middle tunnel of the same normalized
nontrivial torus knot.

Consider first a lift-$\lambda$ splitting using $\gamma_m$ and a
drop-$\lambda$ splitting using $\gamma_n$ applied to $T_{p+r,q+s}$ to
produce the same tunnel. Equating the expressions for their slopes from
Proposition~\ref{prop:sigma_slopes}, we obtain $1/n-1/m=2ps - 2qr = 2$, so
$m=-1$ and $n=1$, giving case (a). Case~(b) is similar.

For a lift-$\lambda$ splitting using $\gamma_m$ and a lift-$\rho$ splitting
using $\gamma_n$, we obtain $2(p+r)(s-q)=1/n-1/m$. The right-hand side can
only be $-2$, $0$, or $2$. Since $T_{p+r,q+s}$ is nontrivial, $p$, $q$,
$r$, and $s$ are all positive, forcing the right-hand side to be $0$ and
hence $m=n$ and $q=s$. Since $ps-qr=1$, we have $q=s=1$ and $p=r+1$, so
$T_{p+r,q+s}=T_{2r+1,2}$. This is case (c)(i). Similar procedures lead to
cases~(c)(ii) and~(c)(iii) (although $ps-qr=1$ must be used earlier in the
calculations making the right-hand side $2+1/n-1/m$), and to no
possibilities for a drop-$\lambda$ and drop-$\rho$ pair.
\end{proof}

\section{The invariants of the Goda-Hayashi-Ishihara tunnel}
\label{sec:GHI_calculation}

For the Goda-Hayashi-Ishihara example described in
Section~\ref{sec:GHI_example}, we found the slope invariants of the first
two cablings to be $[2/3]$ and $-5$. We can now find the slope of the
middle tunnel produced by the lift-$\lambda$ splitting applied to
$T_{3,-4}$. We have
\[M_{3,-4}=L^2UM_{1,1}=\begin{pmatrix}1&-1\\2&-3\end{pmatrix}\ .\]
By Proposition~\ref{prop:sigma_slopes}(b),
the slope of the tunnel produced by
the lift-$\lambda$ splitting with $n=-1$ is $2(-3)(1+2)+1/(-1)=-19$.

\begin{figure}
\begin{center}
\labellist
\small\pinlabel $[2/3]$ at 52 136
\small\pinlabel $-5$ at 61 100
\small\pinlabel $-19$ at 88 65
\endlabellist
\includegraphics[width=0.24\textwidth]{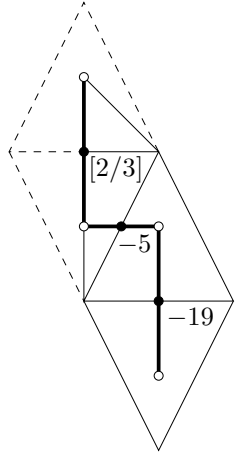}
\caption{The principal path of the Goda-Hayashi-Ishihara tunnel.}
\label{fig:GHI_path}
\end{center}
\end{figure}
The principal path of the tunnel is shown in Figure~\ref{fig:GHI_path}.
As noted in Section~\ref{sec:GHI_example}, the
first two nontrivial cablings in the cabling sequence
are the $L$-constructions that are the first
two steps where the path moves down and to the right.
The $L$-constructions
replace $\lambda$, and the lift-$\lambda$-splitting replaces $\rho$, so the path
turns downward for the $\lambda$-splitting. This proves that the tunnel is
not semisimple (as its principal pair does not contain a primitive disk, or
alternatively because its depth $2$ is greater than $1$). Summarizing, we have
\begin{theorem}\label{thm:GHI}
The Goda-Hayashi-Ishihara tunnel has slope invariant sequence $[2/3]$,
$-5$, $-19$, and binary invariant sequence $1$. It is a regular tunnel of
depth $2$ with the principal path shown in Figure~\ref{fig:GHI_path}.
\end{theorem}
\noindent Of course, for the other examples obtained by varying $n$, the
only difference is that the third slope is $-18+1/n$.

We can also carry out the calculation using the normalized description of
the mirror-image knot given in Section~\ref{sec:GHI_example}.  Start with
$T_{1,1}$ and perform an $L$-construction followed by two $U$-constructions
with slopes $[1/3]$ and $5$ to obtain $T_{4,3}$ with $M_{4,3}=U^2L=
\displaystyle\begin{pmatrix}3&2\\1&1\end{pmatrix}$. We have
$K_\rho=T_{3,2}$ and $K_\lambda=T_{1,1}$. Now, the drop-$\rho$ splitting
with $n=1$ gives the middle tunnel of the mirror image
Goda-Hayashi-Ishihara knot, and by Proposition~\ref{prop:sigma_slopes}(c),
its slope is $2\cdot 3\cdot (2+1)+1/1=19$.

\section{Upper and lower tunnels}
\label{sec:uuperlower}

The $(1,1)$-positions of the knots obtained by the splitting constructions
are readily described using the methods of~\cite{CMsemisimple}. This section
assumes a basic knowledge of that paper.

We will examine the drop-$\rho$ case, the others being very straightforward
modifications. Using the methodology of~\cite{CMsemisimple} and the
associated software~\cite{CMsoftware} that implements its algorithms, we
will find the slopes of the upper and lower tunnels of the
Goda-Hayashi-Ishihara knot.

Let $\omega(a,b)$ denote the braid word describing the torus knot
$T_{a,b}$, as given in~\cite[Section 11]{CMsemisimple}. Begin with a drop-$\rho$ disk dropped onto a horizontal
level, creating the setup picture in Figure~\ref{fig:first_general}(a). We
have $K_U=K_\tau=T_{p+r,q+s}$, with braid word description
$\omega(p+r,q+s)$, and $K_L=K_\rho=T(p,q)$, with braid word description
$\omega(p,q)$. The two vertical arcs are untwisted, and $K_\lambda$ is in
$(1,1)$-position described by the braid word
$\omega(p+r,q+s)\,\omega(p,q)^{-1}$. This equals $\omega(r,s)$ in the
(reduced) braid group $\B$, reflecting the fact that
$K_\lambda=T_{r,s}$. Replacing $\rho$ by $\gamma_n$ creates
$K_{\gamma_n}$, and the position is described by the braid word
$\omega(p+r,q+s)\,\sigma^n\,\omega(p,q)^{-1}$.
From this, the general algorithm in~\cite{CMsemisimple} gives the sequence
of slope invariants.

Let us do the calculations for the Goda-Hayashi-Ishihara examples. We will
use the normalized version, producing the mirror-image examples by the
drop-$\rho$ splitting applied to $T_{4,3}$. We have $K_\tau=T(4,3)$,
$K_\rho=T(3,2)$, and $K_\lambda=T(1,1)$. We compute $\omega(4,3)$ and $\omega(3,2)$:
\smallskip

\noindent\texttt{Semisimple> print fullTorusBraidWord(4,3)}\\
\texttt{l -1 m 1 l -1 m 1 l -2 m 1}
\smallskip

\noindent\texttt{Semisimple> print fullTorusBraidWord(3,2)}\\
\texttt{l -1 m 1 l -2 m 1}
\smallskip

\noindent The knot with position described by
$\omega(4,3)\,\omega(3,2)=\omega(7,5)$ is $T(7,5)$, while the one described
by $\omega(4,3)\omega(3,2)^{-1}=\omega(1,1)$ is the trivial knot
$K_\lambda=T(1,1)$ that would result from a drop-$\rho$ construction with
no twisting (that is, $n=0$). To confirm these, we compute:
\smallskip

\noindent\texttt{Semisimple> upperSlopes( 'l -1 m 1 l -1 m 1 l -2 m 1 s 0 l -1 m 1 l -2 m 1' )}\\
\texttt{[ 1/3 ], 5, 9, 11}
\smallskip

\noindent\texttt{Semisimple> torusUpperSlopes(7,5)}\\
\texttt{[ 1/3 ], 5, 9, 11}
\smallskip

\noindent while entering \texttt{upperSlopes( 'l -1 m 1 l -1 m 1 l -2 m 1 s
0 m -1 l 2 m -1 l 1')} produces empty output, indicating the trivial
knot. For the Goda-Hayashi-Ishihara knot, we insert $\sigma$ giving
\[\omega(4,3)\cdot \sigma \cdot \omega(3,2)^{-1}\]
as a braid word describing its $(1,1)$-position. We find
\smallskip

\noindent\texttt{Semisimple> upperSlopes( 'l -1 m 1 l -1 m 1 l -2 m 1 s 1 l -1 m 1 l -2 m 1' )}\\
\texttt{[ 1/3 ], 7, 9, 11}
\smallskip

\noindent\texttt{Semisimple> lowerSlopes( 'l -1 m 1 l -1 m 1 l -2 m 1 s 1 l -1 m 1 l -2 m 1' )}\\
\texttt{[ 1/3 ], 5, 7, 7, 9}
\smallskip

\noindent We know that using $n=-1$ would give $T(7,5)$, confirmed by
\smallskip

\noindent\texttt{Semisimple> upperSlopes( 'l -1 m 1 l -1 m 1 l -2 m 1 s -1 l -1 m 1 l -2 m 1' )}\\
\texttt{[ 1/3 ], 5, 9, 11}
\smallskip

\noindent We can also observe the effect of changing the number of twists
in the Goda-Hayashi-Ishihara example:
\smallskip

\noindent\texttt{Semisimple> upperSlopes( 'l -1 m 1 l -1 m 1 l -2 m 1 s 2 l -1 m 1 l -2 m 1' )}\\
\texttt{[ 1/3 ], 13/2, -3, -1}
\smallskip

\noindent\texttt{Semisimple> upperSlopes( 'l -1 m 1 l -1 m 1 l -2 m 1 s 3 l -1 m 1 l -2 m 1' )}\\
\texttt{[ 1/3 ], 19/3, 9, 11}
\smallskip

\noindent\texttt{Semisimple> upperSlopes( 'l -1 m 1 l -1 m 1 l -2 m 1 s 4 l -1 m 1 l -2 m 1' )}\\
\texttt{[ 1/3 ], 25/4, -3, -1}
\smallskip

\noindent\texttt{Semisimple> upperSlopes( 'l -1 m 1 l -1 m 1 l -2 m 1 s 5 l -1 m 1 l -2 m 1' )}\\
\texttt{[ 1/3 ], 31/5, 9, 11}
\smallskip

We do not know whether these knots have additional $(1,1)$-positions,
although it seems highly unlikely.

Braid word descriptions for the other three kinds of splittings are
obtained simply by using the appropriate knots for $K_U$ and~$K_L$.

\bibliographystyle{amsplain}

\end{document}